\documentclass[leqno,twoside,11pt]{amsart}

\usepackage{latexsym, esint, color}

\theoremstyle{plain}

\def\endproof{\hspace*{\fill}\mbox{\ \rule{.1in}{.1in}}\medskip }

\newtheorem{theorem}{Theorem}[section]

\newtheorem{lemma}[theorem]{Lemma}
\newtheorem{proposition}[theorem]{Proposition}

\theoremstyle{definition}

\newcommand{\R}{\mathbb{R}}
\renewcommand{\d}{\partial}
\renewcommand{\div}{{\rm div}\,}

\numberwithin{equation}{section}

\begin{document}
\title[stress-assisted diffusion]
{A local and global well-posedness results for \\ the general stress-assisted
  diffusion systems}
\author{Marta Lewicka}
\address{Marta Lewicka,  University of Pittsburgh, Department of Mathematics, 
301 Thackeray Hall, Pittsburgh, PA 15260, USA }
\email{lewicka@pitt.edu}
\author{Piotr B. Mucha}
\address{Piotr B. Mucha, Institute of Applied Mathematics and Mechanics, 
 University of Warsaw, ul. Banacha 2, 02097 Warszawa, Poland}
\email{p.mucha@mimuw.edu.pl}
\date{\today}

\begin{abstract}
We prove the local and global in time existence of the classical
solutions to two general classes of the stress-assisted diffusion systems. Our results
are applicable in the context of the
non-Euclidean elasticity and liquid crystal elastomers.
\end{abstract}

\maketitle

\section{Introduction and the main results.}

There are a number of phenomena where inhomogeneous and incompatible
pre-strain is observed in $3$-dimensional bodies. Growing leaves, gels subjected to
differential swelling, electrodes in electrochemical cells, edges of torn plastic sheets are but a
few examples \cite{12a, 23a, 24a, 28a}. It has also been recently suggested that such
incompatible pre-strains may be exploited as means of actuation of 
micro-mechanical devices \cite{18, 19}. The mathematical foundations
for these theories has lagged behind but has recently been the focus
of much attention. While the static theory involving thin structures
such as pre-strained plates and shells is now reasonably well
understood \cite{13a, 3a, lepa, BLS, lemapa2}, leading to the variationally reduced models constrained to
appropriate types of isometries \cite{lepa, BLS, 4a}, and requiring bringing
together the differential geometry of surfaces with the theory of elasticity appropriately
modified \cite{lemapa1, LOP, 13a, 14a}, the parallel evolutionary PDE model seems to not have been
considered in this context.

\medskip

\subsection{The model and the main results.}

In this paper, we are concerned with two systems of coupled PDEs in the
description of stress-assisted diffusion. The first system:
\begin{equation}\label{maineq}
\left\{\begin{split}
& u_{tt} - \mbox{div}\Big(\partial_F W(\phi, \nabla u)\Big) = 0 \\
& \phi_t = \Delta \Big(\partial_\phi W(\phi,\nabla u)\Big).
\end{split}\right.
\end{equation}
consists of a balance of linear momentum in the deformation field 
$u:\mathbb{R}^3\times\mathbb{R_+}\to \mathbb{R}^3$, and the diffusion
law of the scalar field $\phi:\mathbb{R}^3\times\mathbb{R_+}\to 
\mathbb{R}$ representing the inhomogeneity factor in the elastic
energy density $W$.
The field $\phi$ may be interpreted as
the local swelling/shrinkage rate in morphoge\-ne\-sis at polymerization, or the localized
conformation in liquid crystal elastomers. 

The second system is a quasi-static approximation of
(\ref{maineq}), in which we neglect the material inertia $u_{tt}$,
consistent with the assumption that 
the diffusion time scale is much larger than the time scale of elastic wave propagation:
\begin{equation}\label{maineq2}
\left\{\begin{split}
& - \mbox{div}\Big(\partial_F W(\phi, \nabla u)\Big) = 0 \\
& \phi_s = \Delta \Big(\partial_\phi W(\phi,\nabla u)\Big).
\end{split}\right.
\end{equation}

In both systems, the deformation $u$ induces the deformation gradient,
and the velocity and velocity gradients, respectively denoted as:
$$F = \nabla u \in\mathbb{R}^{3\times 3}, \quad
v=\xi_t\in\mathbb{R}^3, \quad Q=\nabla\xi_t = \nabla v =
F_t\in\mathbb{R}^{3\times 3}.$$
We will be concerned with the local in time well-posedness of
the classical solutions to (\ref{maineq}), and the global
well-posedness of (\ref{maineq2}), subject to the (subset of) initial data:
\begin{equation}\label{initial1}
u(0,\cdot) = u_0, \quad u_t(0,\cdot)=u_1 \quad \mbox{ in } 
\mathbb{R}^3,
\end{equation}
\begin{equation}\label{initial2}
\phi(0,\cdot) = \phi_0 \quad \mbox{ in } \mathbb{R}^3,
\end{equation}
and the non-interpenetration ansatz: 
\begin{equation}\label{nonin}
\det \nabla u > 0 \quad \mbox{ in } \mathbb{R}^3.
\end{equation}

The main results of this paper are the following:
\begin{theorem}\label{th1}
Let $u_0 - \mathrm{id} \in H^4(\R^3)$, $u_1 \in H^3(\R^3)$ and
$\phi_0\in H^3(\R^3)$.  Assume that $W$ is as in subsection \ref{ener}.
Fix $T>0$, and assume that the following quantities:
\begin{equation}\label{small}
\|u_1,\nabla u_0 - \mathrm{Id}_3,\phi_0\|_{H^3}^2 + \|u_0 - \mathrm{id}\|_{L^2}^2 +
\int_{\R^3} W(\phi_0,\nabla u_0)~\mathrm{d}x 
\end{equation}
are sufficiently small in comparison with $T$, and with the constant
$\gamma$ in (\ref{ass}). Then there exists a unique solution $(u, \phi)$
of the problem (\ref{maineq}) (\ref{initial1} - \ref{nonin}), defined on the time
interval $[0,T]$, and such that:
\begin{equation*}
\begin{split} 
& u-\mathrm{id} \in L^\infty(0,T;H^4(\mathbb{R}^3)), \quad u_{tt}\in
L^\infty(0,T;H^2(\R^3)), \\
& \phi \in L^\infty(0,T;H^3(\mathbb{R}^3)) \mbox{ \ \ and \ \ } \phi_t \in
L^2(0,T;H^2(\mathbb{R}^3)).
\end{split}
\end{equation*}
\end{theorem}

\begin{theorem}\label{th2} 
Let $\phi_0 \in H^2(\mathbb{R}^3)$ and assume that $W$ is as in
subsection \ref{ener}.
Assume that $\|\phi_0\|_{H^2}$ is sufficiently small.
Then there exists a unique global in time solution $(u, \phi)$ to
(\ref{maineq2}) (\ref{initial2}) (\ref{nonin}) such that:
\begin{equation*}
\begin{split}
& u-\mathrm{id} \in L^\infty(\mathbb{R}_+;L^6(\mathbb{R}^3)),
\quad \nabla^2 u \in L^2(\mathbb{R}_+;H^2(\mathbb{R}^3)), \\
& \phi \in L^\infty(\mathbb{R}_+; H^2(\mathbb{R}^3)) \mbox{\ \ and \ \ } \nabla
\phi\in L^2(\mathbb{R}_+; H^2(\mathbb{R}^3)).
\end{split}
\end{equation*}
\end{theorem}

\medskip

The proof of Theorem \ref{th1} relies on controlling the energy:
$$ \int_{\R^3} \frac 12 |u_t|^2 +  W(\phi,\nabla u) ~\mbox{d}x, $$
where the hyperbolic character of the first equation in (\ref{maineq})
suggests to seek the a-priori bounds on higher norms of $u$ and
$\phi$ by the standard energy techniques. A detailed analysis reveals
that the special structure of coupling in the stress-assisted diffusion system indeed allows for cancellation of
those terms that otherwise prevent closing the bounds in each of the
two equations in (\ref{maineq}) alone. These terms are displayed in
formulas (\ref{dodici}) and (\ref{dicianove}) in the proof of Lemma
\ref{lemapriori}. Existence of solutions in Theorem \ref{th1} is then
shown via Galerkin's method, where we check that solutions to all
appropriate $\epsilon$-approximations of the original system
(\ref{maineq}) still enjoy the same a-priori bounds in Theorem
\ref{thap}. This is carried out in section \ref{sec3}, while
uniqueness of solutions is proved in section \ref{sec4}.

The proof of Theorem \ref{th2}, given in section \ref{sec5},
is based on the $L^2$-approach as well. The system (\ref{maineq2})
is of elliptic-parabolic type, thus there is no loss of regularity with
respect to the initial data (in contrast to (\ref{maineq})). The
analysis here is simpler than for (\ref{maineq}) and we are able to
show the global in time existence of small solutions.
The toolbox we use for the proofs of both results is universal
for hyperbolic-parabolic and elliptic-parabolic systems. Similar
methods have been applied in \cite{BaMa,FMNP,MPZ,PZ,ZO} to study 
models of elasticity and their couplings with flows of complex
fluids. A key element in these methods is 
the basic conservation law of energy and entropy type.

\medskip

\subsection{The energy density $W$.} \label{ener}
We now introduce the assumptions on the inhomogeneous elastic energy
density $W$ in (\ref{maineq}). Namely,
the nonnegative scalar field $W: \mathbb{R}_+\times
\mathbb{R}^{3\times 3}\rightarrow \overline{\mathbb{R}}_+$ is assumed
to be $\mathcal{C}^4$ in a neighborhood of $(0, \mbox{Id}_3)$ and to
satisfy, with some constant $\gamma >0$:
\begin{equation}\label{ass}
\begin{split}
& W(0, \mbox{Id}) = 0, \quad DW(0, \mbox{Id}) = 0, ~~~ \mbox{ and: } \\ & D^2W(0,\mbox{Id}) :
(\tilde\phi, \tilde F)^{\otimes 2} \geq \gamma (|\tilde\phi|^2 + |\mbox{sym~}\tilde F|^2) \quad
\mbox{ \ \ for all \ }  (\tilde \phi, \tilde F)\in \mathbb{R}\times
\mathbb{R}^{3\times 3}.
\end{split}
\end{equation}

The two main examples of $W$ that we have in mind, concern  non-Euclidean
elasticity and liquid crystal elastomers, where respectively:
\begin{equation}\label{exa}
\begin{split}
W_1(\phi,F) = W_0(FB(\phi)) + \frac{1}{2}|\phi|^2,\\
W_2(\phi,F) = W_0(B(\phi)F) + \frac{1}{2}|\phi|^2.
\end{split}
\end{equation}
are given in terms of the homogeneous energy density
$W_0:\mathbb{R}^{3\times 3}\rightarrow \overline{\mathbb{R}}_+$ and
the smooth tensor field $B:\mathbb{R}\to\mathbb{R}^{3\times 3}$.
In both cases, we assume that $B(\phi)$ is symmetric and positive definite, and that
$B(0)=\mbox{Id}$. 
Further, the principles of material frame invariance, material
consistency, normalisation, and non-degeneracy 
impose the following conditions on $W_0$, valid for all
$F\in\mathbb{R}^{3\times 3}$ and all $R\in SO(3)$:
\begin{equation}\label{elastic_dens}
\begin{minipage}{14cm}
\begin{itemize}
\item[(i)] $W_0(RF) = W_0(F).$
\item[(ii)] $W_0(F)\to +\infty \quad \mbox{ as } \det F\to 0$.
\item[(iii)] $W_0(\mbox{Id}) = 0$.
\item[(iv)] $W_0(F) \geq c~\mbox{dist}^2(F, SO(3))$.
\end{itemize}
\end{minipage}
\end{equation}
Examples of  $W_0$ satisfying the above conditions are: 
\begin{equation*}
\begin{split} 
W_{0,1}(F) & = |(F^TF)^{1/2} - \mbox{Id}|^2 + |\log \det F|^q \\
W_{0,2}(F) & =  |(F^TF)^{1/2} - \mbox{Id}|^2 + \left|\frac{1}{\det F} - 1\right|^q 
\mbox{ for } \det F>0,
\end{split}
\end{equation*}
where $q>1$ and $W_{0,i}$ is intended to be $+\infty$ if $\det F\leq 0$ \cite{MS}.
Another case-study  example, satisfying (i), (iii) but not (iv) is:
$W_0(F)=|F^TF-\mbox{Id}|^2$.

We have the following observation, which we will prove in the Appendix:

\begin{proposition}\label{prop}
For $W_0$ which is $\mathcal{C}^2$ in a neighborhood of $SO(3)$ and
$B$ which is $\mathcal{C}^2$ in a neighborhood of $0$, assume 
(\ref{elastic_dens}) and assume that $B(0) = \mathrm{Id}$. Then $W_1$
and $W_2$ in (\ref{exa}) satisfy (\ref{ass}).
\end{proposition}


\medskip

\subsection{Background and relation to previous works.}
To put our results in a broader context, consider a general referential
domain $\Omega$ which is an open, smooth and simply connected subset
of $\mathbb{R}^3$. Let $G:\bar{\Omega}\rightarrow \mathbb{R}^{3\times 3}$ be a given smooth
Riemann metric on $\Omega$ and denote its unique positive definite symmetric
square root by $B=\sqrt{G}$. The ``incom\-pa\-ti\-ble elastic
energy'' of a deformation $u$ of $\Omega$ is then given by:
\begin{equation}\label{functiona}
E(u, \Omega) = \int_{\Omega} W_0(\nabla u(x) B(x)^{-1})~\mbox{d}x \qquad \forall
u\in W^{1,2}(\Omega,\mathbb{R}^3),
\end{equation}
where the elastic energy density $W_0$ is as in (\ref{elastic_dens}).
It has been proved in \cite{lepa} that: 
$$\inf_{u\in W^{1,2}(\Omega, \mathbb{R}^3)} E(u, \Omega) = 0$$ 
if and only if the Riemann curvature tensor of $G$ vanishes
identically in $\Omega$ and when (equivalently) the infimum above is achieved
through a smooth isometric immersion $u$ of $G$. 

It is worth mentioning that in the context of thin films when
$\Omega = \Omega^h = U\times (-\frac{h}{2}, \frac{h}{2})$ with some $U\subset
\mathbb{R}^2$, there is a large body of literature relating the
magnitude of curvatures of $G$ to the scaling of $\inf E(\cdot, \Omega^h)$ in terms of
the film's thickness $h$, and subsequently deriving the residual
$2$-dimensional energies using the variational techniques.
Firstly, in the Euclidean case of $G=\mbox{Id}_3$, where the residual energies are driven by 
presence of applied forces $f^h\sim h^\alpha$, three
distinct limiting theories have been obtained \cite{FJMhier} for $\frac{1}{h}E(\cdot,
\Omega^h)\sim h^\beta$
with $\beta>2$ (equivalently $\alpha>2$). Namely: $\beta\in (2,4)$ corresponded to the
linearized Kirchhoff model (nonlinear bending energy), $\beta=4$ to the classical von-K\'arm\'an
model, and $\beta>4$ to the linear elasticity. For $\beta=0$ the
membrane energy has been derived in the seminal papers \cite{LR1,
  LR2}, while the case $\beta=2$ was considered in \cite{FJMM}.
Secondly, in \cite{LPhier} a higher order (infinite) hierarchy of scalings and of the resulting elastic
theories of shells, where the reference configuration is a thin curved
film, has been derived by an asymptotic calculus.

Thirdly, in the context of the non-Euclidean energy (\ref{functiona}),
it has been shown in \cite{BLS} that the scaling: $\inf
\frac{1}{h} E(\cdot,\Omega^h)\sim h^2$ only occurs
when the metric $G_{2\times 2}$ on the mid-plate $U$ can be
isometrically immersed in $\mathbb{R}^3$ with the regularity $W^{2,2}$
and when, at the same time, the three appropriate Riemann
curvatures of $G$ do not vanish identically; the relevant residual theory, obtained through
$\Gamma$-convergence, yielded then a Kirchhoff-like residual energy.
Further, in \cite{LRR} the authors proved that the only outstanding nontrivial residual theory is
a von K\`arm\`an-like energy, valid when: $\inf \frac{1}{h}E(\cdot, \Omega^h) \sim h^4$.
This scale separation, contrary to \cite{FJMhier, LPhier}, is due to the fact that while the
magnitude of external forces is adjustable at will, it seems not to be
the case for the interior mechanism of a given metric
$G$ which does not depend on $h$. In fact, it is the curvature tensor of $G$ 
which induces the nontrivial stresses in the thin film and it has only
six independent components, namely the six sectional
curvatures created out of the three principal directions, which further
fall into two categories: including or excluding the thin direction
variable. The simultaneous vanishing of curvatures in each of these categories
correspond to the two scenarios at hand in terms of the scaling of the
residual energy. 

Other types of the residual energies, pertaining to different contexts
and scalings, have been studied and derived by the authors in \cite{3a, 4a, 12a, 13a,
  14a, lemapa1, lemapa2, lemapa2new, LOP, 23a}.  

\smallskip

Note that, at the formal level, the Euler-Lagrange equations of
(\ref{functiona}) are precisely the first equation in the system (\ref{maineq}).
The dynamical viscoelasticity has been the subject of
vast studies in the last decades (see for example \cite{AM, 4, 1, Demoulini, BLZ, 11,
  25} and references therein), where various results on existence, asymptotics and stability have been
obtained for a large class of models. For the coupled systems of
stress-assisted diffusion of the type (\ref{maineq}), we found a substantial
body of literature in the Applied Mechanics community \cite{v1, v2, v3, v4,
  v5}, deriving these equations from basic principles of continuum
mechanics and irreversible thermodynamics.  For example, the system derived in
\cite{v5} is quite close to (\ref{maineq2}) from the view point of
theory of PDEs; indeed the structure of nonlinearity in both systems
is almost the same. However, derivation from the first principles
aside,  it seems that the analytical study of the Cauchy problem, particularly in long temporal
ranges, has not been yet carried out. The closest investigation in
this direction has been recently proposed in \cite{JiWa},
concerning existence of solutions for models of nonlinear
thermoelasticity, and in \cite{ZO} where the authors 
examine further models of thermoviscoelasticity from the viewpoint of mathematical well-posedness. 
We refer here to \cite{RaSh,LeMucha} as well.

\medskip

\subsection{Notation.}
Throughout the paper we use the following notation.
In (\ref{maineq}) the operator ${div}$ stands for the spacial divergence of an appropriate field.
We use the convention that the divergence of a matrix field is taken
row-wise. We use the matrix norm $|F|=(\mbox{tr}(F^TF))^{1/2}$, which
is induced by the inner product: $\langle F_1:F_2\rangle = \mbox{tr}(F_1^TF_2)$.

The derivatives of $W$ are denoted by $DW$, $D^2W$ etc, while their
action on the appropriate variations $(\tilde\phi, \tilde
F)\in\mathbb{R}\times \mathbb{R}^{3\times 3}$ is denoted by:
$DW(\phi,\nabla u) : (\tilde \phi, \tilde F)$, $D^2W(\phi,\nabla u)
: (\tilde \phi, \tilde F)^{\otimes 2}$ etc, often abbreviating to $DW : (\tilde \phi, \tilde F)$ and $(D^2W)
: (\tilde \phi, \tilde F)^{\otimes 2}$ when no confusion arises.

The partial derivative of $W$ with respect to its second argument is
denoted by $\partial_FW\in \mathbb{R}^{3\times 3}$. The derivative in
the direction of the variation $\tilde F\in \mathbb{R}^{3\times 3}$ is
then $\langle (\partial_FW):\tilde F\rangle\in\mathbb{R}$. 
By $(\partial^k_FW):(\tilde F_1\otimes \tilde F_2\ldots \otimes \tilde
F_{k-1})\in\mathbb{R}^{3\times 3}$ we denote the linear map acting on
$F\in\mathbb{R}^{3\times 3}$ as the $k$th derivative of
$W$ in the direction of $\tilde F_1, \tilde F_2\ldots \tilde
F_{k-1}, F$. Hence, differentiating in $F$ gives:
$$(\partial^k_FW):(\tilde F_1\otimes \ldots \tilde F_k) = \Big\langle
\big((\partial^k_FW):(\tilde F_1\otimes \ldots \tilde F_{k-1})\big) :
\tilde F_k\Big\rangle \in\mathbb{R}.$$ 
Finally, $C, c>0$ stand for universal constants, independent of the
variable quantities at hand. 

\medskip

\subsection{Acknowledgments.}
M.L. was partially supported by the NSF grant DMS-0846996 
and the NSF grant DMS-1406730. 
P.B.M. was partly supported by the NCN grant No. 2011/01/B/ST1/01197.

\section{The crucial a priori estimate.} \label{abd}

\begin{lemma}\label{gaga5}
Every solution $(u,\phi)$ to (\ref{maineq}), with regularity prescribed in Theorem \ref{th1}, satisfies:
\begin{equation}\label{nove}
\frac{\mathrm{d}}{\mathrm{d}t} \int_{\mathbb{R}^3} \frac{1}{2} |u_t|^2 +
W(\phi, \nabla u)~\mathrm{d}x + \int|\nabla (- \Delta)^{-1}\phi_t|^2
~\mathrm{d}x = 0.
\end{equation}
\end{lemma}
\begin{proof}
Testing the first equation in (\ref{maineq}) by $u_t$ and integrating by
parts gives:
\begin{equation*}
\begin{split}
\frac{\mathrm{d}}{\mathrm{d}t} \int_{\mathbb{R}^3} \frac{1}{2} |u_t|^2 +
& W(\phi, \nabla u)~\mathrm{d}x = \int \langle u_t, u_{tt}\rangle
+ \partial_\phi W(\phi,\nabla u) \phi_t
+\langle \partial_FW(\phi,\nabla u) : \nabla u_t\rangle ~\mathrm{d}x\\
& = \int \langle u_t, \mbox{div}\left(\partial_FW(\phi,\nabla u)\right)\rangle 
+ \langle \partial_FW(\phi,\nabla u) : \nabla u_t\rangle
+ \partial_\phi W(\phi,\nabla u) \phi_t ~\mathrm{d}x \\ & =
\int_{\mathbb{R}^3} \partial_\phi W(\phi,\nabla u) \phi_t ~\mathrm{d}x. 
\end{split}
\end{equation*}
Define $\psi=(-\Delta)^{-1}\phi$ and integrate the second equation in
(\ref{maineq}) against $\psi_t$:
\begin{equation*}
\int_{\mathbb{R}^3} |\nabla\psi_t|^2 ~\mathrm{d}x = \int \phi_t \psi_t ~\mathrm{d}x
= \int \psi_t \Delta \big(\partial_\phi W(\phi, \nabla u)\big)~\mathrm{d}x = - 
\int_{\mathbb{R}^3} \partial_\phi W(\phi,\nabla u) \phi_t ~\mathrm{d}x. 
\end{equation*}
Summing the above two equalities yields (\ref{nove}) and achieves the proof.
\end{proof}

\medskip

For every $i,j,k\in\{1,2,3\}$ we now define the correction terms:
\begin{equation}\label{R2}
\begin{split}
\mathcal{R}_{ijk} =  & (\partial_\phi\partial_F^2W) : \big(\nabla u_{x_i,
  x_j}\otimes \nabla u_{x_k} + \nabla u_{x_i, x_k}\otimes \nabla u_{x_j} + \nabla u_{x_j,
  x_k}\otimes \nabla u_{x_i} \big) \\ &
+ (\partial_\phi^2\partial_F\partial_\phi W) : \big(\nabla u_{x_i,
  x_j}\phi_{x_k} + \nabla u_{x_i, x_k}\phi_{x_j} +\nabla u_{x_j,
  x_k}\phi_{x_i} \\ & \qquad\qquad \qquad\qquad \qquad\qquad +
\nabla u_{x_i}\phi_{x_j, x_k} + \nabla u_{x_j}\phi_{x_i, x_k}+ \nabla u_{x_k}\phi_{x_i, x_j}\big) 
\\ & + (\partial_\phi\partial_F^3W) :
\nabla u_{x_i}\otimes \nabla u_{x_j} \otimes \nabla u_{x_k} \\ & +
(\partial^2_\phi\partial_F^2 W) : \big(\nabla u_{x_i} \otimes \nabla
u_{x_j} \phi_{x_k} + \nabla u_{x_i} \otimes \nabla u_{x_k} \phi_{x_j} +\nabla u_{x_j} \otimes \nabla
u_{x_k} \phi_{x_1}   \big)
\\ & + (\partial^3_\phi\partial_F W) : \big(\nabla u_{x_i}
\phi_{x_j} \phi_{x_k} + \nabla u_{x_j} \phi_{x_i} \phi_{x_k} +\nabla u_{x_k}
\phi_{x_i} \phi_{x_j}\big)\\ &
+ (\partial_\phi^4 W) \phi_{x_i} \phi_{x_j}\phi_{x_k} \\ & +
(\partial^3_\phi W) \big(\phi_{x_i, x_j} \phi_{x_k} +  \phi_{x_j, x_k}
\phi_{x_i} +\phi_{x_i, x_k} \phi_{x_j}\big).
\end{split}
\end{equation}

\begin{lemma}\label{lemapriori}
Let $(u,\phi)$ be a solution to (\ref{maineq}), with regularity
prescribed in Theorem \ref{th1}. For $t>0$, define the two quantities:
\begin{equation*}
\begin{split}
& \mathcal{E}(t) = \int_{\mathbb{R}^3} |u_t|^2 + |\nabla^3u_t|^2  +
2W(\phi,\nabla u) \\ & \qquad \qquad + \sum_{i,j,k=1..3} D^2W(\phi,\nabla u) : (\phi_{x_i, x_j,
  x_k}, \nabla u_{x_i, x_j, x_k})^{\otimes 2} 
+ 2 \sum_{i,j,k=1..3} \mathcal{R}_{ijk}\phi_{x_i, x_j, x_k} ~\mathrm{d}x,
\\ & \mathcal{Z}(t) = \|u_t\|^2_{H^{3}(\mathbb{R}^3)} +
\|\nabla u - \mathrm{Id}\|^2_{H^{3}(\mathbb{R}^3)}  + \|\phi\|^2_{H^{3}(\mathbb{R}^3)}. 
\end{split}
\end{equation*}
Then:
\begin{equation}\label{ventinove}
\frac{\mathrm{d}}{\mathrm{d}t} \mathcal{E}\leq C(\mathcal{Z}^{4} + \mathcal{Z}^{3/2}).
\end{equation}
\end{lemma}
\begin{proof}
{\bf 1.} We differentiate the first equation in (\ref{maineq}) in a spacial
direction $x_i\in \{x_1, x_2, x_3\}$:
$$u_{x_i,tt} - \mbox{div}\left(\partial_F^2W(\phi, \nabla
  u) : \nabla u_{x_i}\right) = \mbox{div}\left(\partial_F\partial_\phi
  W(\phi, \nabla u)\phi_{x_i}\right). $$
We now differentiate the above twice more in the directions $x_i,
x_j\in \{x_1, x_2, x_3\}$:
\begin{equation}\label{undici}
u_{x_i, x_j, x_k, tt} - \mbox{div}\left(\partial_F^2W(\phi, \nabla
  u) : \nabla u_{x_i, x_j, x_k}\right) = \mbox{div}\left(\partial_F\partial_\phi
  W(\phi, \nabla u)\phi_{x_i, x_j, x_k}\right) + \mathcal{R}_1. 
\end{equation}
The error term $\mathcal{R}_1$ above has the following form, where we suppress
the distinction between different $x_i, x_j, x_k$, retaining hence
only the structure of different terms:
\begin{equation}\label{R1}
\begin{split}
\mathcal{R}_1 = \mbox{div}\Big(&3(\partial_F^3W) : \nabla u_x\otimes \nabla u_{xx}
+ 3(\partial_F^2\partial_\phi W) : (\nabla u_{xx}\phi_x +\nabla
u_x\phi_{xx}) \\ & 
+ 3(\partial_F\partial^2_\phi W)\phi_x\phi_{xx} + (\partial_F^4W) : 
(\nabla u_x)^{\otimes 3} + 3(\partial_F^3\partial_\phi W) : (\nabla
u_x)^{\otimes 2} \phi_x \\ & + 3(\partial_F^2\partial^2_\phi W) : \nabla u_x
(\phi_x)^2 + (\partial_F\partial_\phi^3 W)(\phi_x)^3 \Big).
\end{split}
\end{equation}
Above and in what follows, we also write $(\partial_F^3W)$ instead of $\partial_F^3W(\phi,
\nabla u)$, and $(\partial_F^2\partial_\phi W)$ instead of
$\partial_F^2\partial_\phi W(\phi, \nabla u)$, etc.
Integrating (\ref{undici}) by parts against $u_{x_i, x_j, x_k, t}$ we get:
\begin{equation}\label{dodici}
\begin{split}
& \frac{\mbox{d}}{\mbox{d}t}\int_{\mathbb{R}^3} \frac{1}{2} (u_{x_i,
  x_j, x_k, t})^2 +\phi_{x_i, x_j, x_k}\langle(\partial_F\partial_\phi
W):\nabla u_{x_i, x_j, x_k}\rangle  \\ & \qquad\qquad\qquad \qquad
\qquad + \frac{1}{2}
(\partial^2_FW) : \nabla u_{x_i, x_j, x_k}\otimes \nabla u_{x_i, x_j,
  x_k} ~\mbox{d}x \\ &
= \boxed{\int_{\mathbb{R}^3}  \phi_{x_i, x_j, x_k, t}\langle(\partial_F\partial_\phi W):\nabla
u_{x_i, x_j, x_k}\rangle ~\mathrm{d}x } \\ & \qquad +
\int_{\mathbb{R}^3}  \phi_{x_i, x_j, x_k}\langle(\partial_t\partial_F\partial_\phi W):\nabla
u_{x_i, x_j, x_k}\rangle ~\mathrm{d}x \\ & \qquad + \frac{1}{2} \int_{\mathbb{R}^3} 
(\partial_t\partial^2_F\partial_\phi W) : \nabla u_{x_i, x_j,
  x_k} \otimes \nabla u_{x_i, x_j, x_k} ~\mathrm{d}x + \int_{\mathbb{R}^3}  \mathcal{R}_1 u_{x_i, x_j, x_k, t} ~\mathrm{d}x.
\end{split}
\end{equation}

\medskip

{\bf 2.} Differentiate now the second equation in (\ref{maineq}) in $x_i\in \{x_1, x_2, x_3\}$:
$$\phi_{x_i, t} = \Delta\big(\langle\partial_\phi\partial_FW (\phi, \nabla
  u) : \nabla u_{x_i}\rangle + \partial^2_\phi W(\phi, \nabla u)\phi_{x_i}\big).$$
As before, differentiate twice more in $x_i, x_j\in \{x_1, x_2,
x_3\}$, to obtain:
\begin{equation}\label{diciotto}
\phi_{x_i, x_j, x_k, t} = \Delta\left((\partial_\phi\partial_F W) : \nabla
  u_{x_i, x_j, x_k} + (\partial^2_\phi W) \phi_{x_i, x_j, x_k}
  +\mathcal{R} _{ijk}\right),
\end{equation}
where $\mathcal{R}$ is given in (\ref{R2}).
Testing (\ref{diciotto}) against $(-\Delta)^{-1}\phi_{x_i, x_j, x_k, t} =
\psi_{x_i, x_j, x_k, t}$, we get:
\begin{equation}\label{dicianove}
\begin{split}
 -\int_{\mathbb{R}^3} |\nabla\psi_{x_i, x_j, x_k, t}|^2&~\mbox{d}x  
=  \boxed{\int_{\mathbb{R}^3}  \phi_{x_i, x_j, x_k, t}\langle(\partial_F\partial_\phi W):\nabla
u_{x_i, x_j, x_k}\rangle ~\mbox{d}x}  \\ & 
+\frac{\mbox{d}}{\mbox{d}t}\int_{\mathbb{R}^3} \frac{1}{2}
(\partial^2_\phi W) (\phi_{x_i, x_j,  x_k})^2 ~\mbox{d}x \\ & - \frac{1}{2} \int_{\mathbb{R}^3} 
(\partial_t(\partial^2_\phi W)) (\phi_{x_i, x_j,  x_k})^2~\mbox{d}x
+ \int_{\mathbb{R}^3}  \mathcal{R}_{ijk} \phi_{x_i, x_j, x_k, t} ~\mbox{d}x  .
\end{split}
\end{equation}
Note now that the first terms in the right hand side of both
(\ref{dodici}) and (\ref{dicianove}), namely the terms displayed in boxes, are the same. Consequently,
subtracting (\ref{dicianove}) from (\ref{dodici}), we get:
\begin{equation}\label{venti}
\begin{split}
 \int_{\mathbb{R}^3} |\nabla&\psi_{x_i, x_j, x_k, t}|^2~\mbox{d}x  \\ & +
\frac{1}{2}\frac{\mbox{d}}{\mbox{d}t}\int_{\mathbb{R}^3} (u_{x_i,
  x_j, x_k, t})^2 + D^2W(\phi, \nabla u) : (\phi_{x_i, x_j, x_k},
\nabla u_{x_i, x_j, x_k})^{\otimes 2} ~\mbox{d}x \\ & 
+ \int_{\mathbb{R}^3}\mathcal{R}_{ijk}\phi_{x_i, x_j, x_k, t} ~\mbox{d}x\\ &
\qquad\qquad  = \int_{\mathbb{R}^3}  \phi_{x_i, x_j, x_k}\langle(\partial_t\partial_F\partial_\phi W):\nabla
u_{x_i, x_j, x_k}\rangle ~\mbox{d}x  \\ & \qquad \qquad\qquad + \frac{1}{2} \int_{\mathbb{R}^3} 
(\partial_t\partial^2_F\partial_\phi W) : \nabla u_{x_i, x_j,
  x_k} \otimes  \nabla u_{x_i, x_j, x_k}~\mbox{d}x \\ & \qquad \qquad \qquad  + \frac{1}{2} \int_{\mathbb{R}^3} 
(\partial_t(\partial^2_\phi W)) (\phi_{x_i, x_j,  x_k})^2 ~\mbox{d}x
+ \int_{\mathbb{R}^3}  \mathcal{R}_1 u_{x_i, x_j, x_k, t}~\mbox{d}x. 
\end{split}
\end{equation}

\medskip

{\bf 3.} We will now estimate terms in the right hand side of
(\ref{venti}) and prove that:
\begin{equation}\label{estim1}
\begin{split}
\int_{\mathbb{R}^3}&  |\phi_{x_i, x_j, x_k}|  |\partial_t\partial_F\partial_\phi W | |\nabla
u_{x_i, x_j, x_k}|~\mbox{d}x   + \int_{\mathbb{R}^3}  |\partial_t\partial^2_F\partial_\phi W| |\nabla u_{x_i, x_j,
  x_k}|^2~\mbox{d}x  \\ & 
\qquad\qquad + \int_{\mathbb{R}^3}  |\partial_t(\partial^2_\phi W)| |\phi_{x_i, x_j,  x_k}|^2 ~\mbox{d}x 
\leq C \left(\mathcal{Z}^{3/2} + \|\nabla
  \psi_t\|_{H^{3}(\mathbb{R}^3)} \mathcal{Z}\right).
\end{split}
\end{equation}
and:
\begin{equation}\label{estim2}
|\int_{\mathbb{R}^3}  \mathcal{R}_1 u_{x_i, x_j, x_k, t}~\mbox{d}x  |\leq C
\left(\mathcal{Z}^{3/2} + \mathcal{Z}^2 + \mathcal{Z}^{5/2}\right) 
+ C \|\nabla \psi_t\|_{H^{3}(\mathbb{R}^3)} \left( \mathcal{Z}^{3/2}
  + \mathcal{Z}^2\right).
\end{equation}

For the first term in (\ref{estim1}), we note that by the Sobolev
embedding $\mathcal{C}^{0, 1/2}(\mathbb{R}^3) \hookrightarrow
H^{2}(\mathbb{R}^3)$ one easily gets:
\begin{equation*}
\begin{split}
\int_{\R^3} |\phi_{x_i, x_j, x_k}|  |\partial_t\partial_F\partial_\phi W | |\nabla
u_{x_i, x_j, x_k}|  & \leq \| (\partial^2_F\partial_\phi W) : \nabla u_t +
(\partial_F\partial^2_\phi W)\phi_t\|_{L^\infty}
\|\nabla^3\phi\|_{L^2} \|\nabla^4u\|_{L^2} \\ & \leq C\left( \|\nabla
  u_t\|_{L^\infty} + \|\phi_t\|_{L^\infty}\right)
\|\nabla^3\phi\|_{L^2} \|\nabla^4u\|_{L^2} \\ & \leq C\left( \|\nabla
  u_t\|_{H^{2}} + \|\Delta\psi_t\|_{H^{2}}\right)
\|\nabla^3\phi\|_{L^2} \|\nabla^4u\|_{L^2}  \\ & 
\leq C \left(\mathcal{Z}^{3/2} + \|\nabla
  \psi_t\|_{H^{3}(\mathbb{R}^3)} \mathcal{Z}\right).
\end{split}
\end{equation*}
Similarly, the other two terms in (\ref{estim1}) are bounded by:
$$C\left( \|\nabla u_t\|_{L^\infty} + \|\phi_t\|_{L^\infty}\right)
\left( \|\nabla^4u\|_{L^2}^2 + \|\nabla^3\phi\|_{L^2}^2\right),$$ 
which implies the same estimate as before.

Regarding (\ref{estim2}), the first term in $\int_{\mathbb{R}^3} \mathcal{R}_1u_{x_i,
  x_j, x_k, t} ~\mbox{d}x $, is bounded by:
\begin{equation*}
\begin{split}
\int_{\mathbb{R}^3} |\mbox{div}\big((\partial_F^3 W) : \nabla
  u_x\otimes \nabla u_{xx}\big)| |\nabla^3u_t|~\mbox{d}x 
& \leq C \Big( (\|\nabla  u_t\|_{L^\infty} + \|\phi_t\|_{L^\infty})
  \|\nabla^2 u\|_{L^\infty}  \|\nabla^3  u\|_{L^2} \\ & \qquad + \|\nabla^3
  u\|_{L^4}^2 + \|\nabla^2  u\|_{L^\infty} \|\nabla^4  u\|_{L^\infty}
\Big) \|\nabla^3  u_t\|_{L^2} \\ & \leq  C \left(\mathcal{Z}^{3/2} +
  \mathcal{Z}^{2} + \|\nabla
  \psi_t\|_{H^{3}(\mathbb{R}^3)} \mathcal{Z}^{3/2}\right),
\end{split}
\end{equation*}
because of the Sobolev embedding $W^{1,2}(\mathbb{R}^3)\hookrightarrow L^p(\mathbb{R}^3)$ 
valid for any $p\in [2,6]$. Also:
\begin{equation*}
\begin{split}
\int_{\mathbb{R}^3} |\mbox{div}\big((\partial_F^4 W) : (\nabla
  u_x)^{\otimes 3}\big)| |\nabla^3u_t|~\mbox{d}x 
& \leq C \Big( (\|\nabla  u_t\|_{L^\infty} + \|\phi_t\|_{L^\infty})
  \|\nabla^2 u\|^3_{L^6}  \\ & \qquad\qquad + \|\nabla^2 u\|_{L^\infty}^2  \|\nabla^3
  u\|_{L^2} \Big) \|\nabla^3  u_t\|_{L^2} \\ & 
\leq  C \left(\mathcal{Z}^{5/2} + \mathcal{Z}^2 + \|\nabla   \psi_t\|_{H^{3}(\mathbb{R}^3)} \mathcal{Z}^2\right).
\end{split}
\end{equation*}
Other terms in $\mathcal{R}_1$ induce the same estimate as above. This establishes (\ref{estim2}).

\medskip

{\bf 4.} We now consider the last term in the right hand side of
(\ref{venti}):
\begin{equation}\label{ma}
\int_{\mathbb{R}^3} \mathcal{R}_{ijk}\phi_{x_i, x_j, x_k, t} ~\mbox{d}x =
\left(\frac{\mbox{d}}{\mbox{d}t}
  \int_{\mathbb{R}^3}\mathcal{R}_{ijk}\phi_{x_i, x_j, x_k}~\mbox{d}x \right) -
\int_{\mathbb{R}^3}(\mathcal{R}_{ijk})_t\phi_{x_i, x_j,
  x_k}~\mbox{d}x. 
\end{equation}
We now prove that:
\begin{equation}\label{estim3}
|\int_{\mathbb{R}^3}  (\mathcal{R}_{ijk})_t \phi_{x_i, x_j, x_k}~\mbox{d}x  |\leq C
\left(\mathcal{Z}^{3/2} + \mathcal{Z}^{5/2}\right) 
+ C \|\nabla \psi_t\|_{H^{3}(\mathbb{R}^3)} \left( \mathcal{Z}^{3/2}
  + \mathcal{Z}^2\right).
\end{equation}
First, using the notational convention as in (\ref{R1}),
$\mathcal{R}$ can be replaced by:
\begin{equation}\label{R3}
\begin{split}
\mathcal{R}_2 =  & ~3(\partial_\phi\partial_F^2W) : \nabla u_x\otimes \nabla u_{xx}
+ 3(\partial_\phi^2\partial_F\partial_\phi W) : (\nabla u_{xx}\phi_x +\nabla
u_x\phi_{xx}) \\ & + (\partial_\phi\partial_F^3W) :
(\nabla u_x)^{\otimes 3} + 3(\partial^2_\phi\partial_F^2 W) : (\nabla
u_x)^{\otimes 2} \phi_x \\ & + 3(\partial_F\partial^3_\phi W) : \nabla u_x
(\phi_x)^2 + (\partial_\phi^4 W)(\phi_x)^3 + 3(\partial^3_\phi W)\phi_x\phi_{xx}. 
\end{split}
\end{equation}
The first term in (\ref{R3}) can be estimated as before, using
embedding and interpolation theorems:
\begin{equation*}
\begin{split}
\int_{\mathbb{R}^3} |\big((\partial_\phi\partial_F^2 &W) : \nabla
  u_x\otimes \nabla u_{xx}\big)_t| |\nabla^3\phi|~\mbox{d}x 
\\ & \leq C \Big( (\|\nabla  u_t\|_{L^\infty} + \|\phi_t\|_{L^\infty})
  \|\nabla^2 u\|_{L^\infty}  \|\nabla^3  u\|_{L^2} \\ & \qquad \qquad + \|\nabla^2
  u_t\|_{L^4}\|\nabla^3 u\|_{L^4} + \|\nabla^2  u\|_{L^\infty} \|\nabla^3  u_t\|_{L^2}
\Big) \|\nabla^3  \phi\|_{L^2} \\ & \leq  C \left(\mathcal{Z}^{3/2} +
  \mathcal{Z}^{2} + \|\nabla
  \psi_t\|_{H^{3}(\mathbb{R}^3)} \mathcal{Z}^{3/2}\right),
\end{split}
\end{equation*}
while the third term in $\mathcal{R}_2$ is estimated by:
\begin{equation*}
\begin{split}
\int_{\mathbb{R}^3} |\big(\partial_\phi\partial_F^3 W) : & (\nabla
  u_x)^{\otimes 3}\big)_t| |\nabla^3\phi|~\mbox{d}x \\
& \leq C \Big( (\|\nabla  u_t\|_{L^\infty} + \|\phi_t\|_{L^\infty})
  \|\nabla^2 u\|^3_{L^6}  + \|\nabla^2 u\|_{L^\infty}^2  \|\nabla^3
  u_t \|_{L^2} \Big) \|\nabla^3  \phi\|_{L^2} \\ & 
\leq  C \left(\mathcal{Z}^{5/2} + \mathcal{Z}^2 + \|\nabla \psi_t\|_{H^{3}(\mathbb{R}^3)} \mathcal{Z}^2\right).
\end{split}
\end{equation*}
Other terms in $\mathcal{R}_2$ induce the same estimate as above. This establishes (\ref{estim3}).

\medskip

{\bf 5.} Summing now (\ref{venti}) over all triples $x_i, x_j, x_k$,
adding (\ref{nove}), and taking into account (\ref{ma}),
(\ref{estim1}),  (\ref{estim2}) and (\ref{estim3}), we obtain: 
\begin{equation*}
\begin{split}
\frac{\mathrm{d}}{\mathrm{d}t} \mathcal{E} +
\left( 2\|\nabla\psi_t\|_{L^2}^2 + \|\nabla^4\psi_t\|^2_{L^2}\right) 
 & \leq C \left(\mathcal{Z}^{5/2} + \mathcal{Z}^{3/2}\right)
 + C \|\nabla \psi_t\|_{W^{3}_{2}(\mathbb{R}^3)} \left(\mathcal{Z}^2 + \mathcal{Z}\right) \\
& \leq \epsilon \|\nabla \psi_t\|_{H^{3}(\mathbb{R}^3)}^2  + C
\left(\mathcal{Z}^{4} + \mathcal{Z}^{3/2}\right). 
\end{split}
\end{equation*}
in view of Young's inequality. Consequently, (\ref{ventinove}) follows
and the proof is complete.
\end{proof}

We now deduce the main a-priori estimate of this section:

\begin{theorem}\label{thap}
Under the assumptions of Theorem \ref{th1}, any solution on the time
interval $[0, T]$ to (\ref{maineq}) (\ref{initial1} - \ref{initial2}) satisfies:
\begin{equation}\label{a12A}
 \sup_{t\leq T} \mathcal{Z}(t) \leq C\Big(\mathcal{E}(0) +
 T^2\mathcal{E}_0(0) + \|u_0-\mathrm{id}\|_{L^2}^2\Big),
\end{equation}
where: $\mathcal{E}_0(0)=\int_{\mathbb{R}^3} |u_1|^2 +
2W(\phi_0,\nabla u_0)~\mathrm{d}x$, and $C$ is a universal constant.
\end{theorem}
\begin{proof}
Assume that the quantities in (\ref{small}) are sufficiently small. In
particular, we require that $\mathcal{Z}(0)\ll 1$ and that
$\mathcal{Z}$ is sufficiently small on an interval $[0, t_0]$, where
we also choose an appropriate $t_0\ll T$. Lemma  \ref{lemapriori} implies: $\mathcal{E}'(t) \leq C
\mathcal{Z}^{3/2}(t)$, which is equivalent to:
\begin{equation}\label{a1}
\mathcal{E}(t) \leq C\int_0^t \mathcal{Z}^{3/2}(s) ~\mbox{d}s + \mathcal{E}(0).
\end{equation}
Further, by (\ref{nove}) it follows that:
\begin{equation}\label{a2}
 \sup_t \|u_t\|_{L^2}^2 \leq \mathcal{E}_0(0).
\end{equation}
Since:
\begin{equation}\label{a3}
\begin{split}
 \forall t\leq t_0\qquad \|u(t)-\mathrm{id}\|_{L^2}^2 & = 2\int_0^t
 \int_{\mathbb{R}^3} \langle u -\mathrm{id},  u_t\rangle ~\mbox{d}x + \|u_0-\mathrm{id}\|_{L^2}^2 \\
& \leq 2T\big(\sup_{s\leq t} \|u_t\|_{L^2}\big) \Big(\sup_{s\leq
  t}\|u-\mathrm{id}|_{L^2}\Big) +  \|u_0-\mathrm{id}\|_{L^2}^2, 
\end{split}
\end{equation}
we easily obtain in view of (\ref{a2}):
\begin{equation}\label{a4}
\begin{split}
\sup_{t\leq t_0} \|u(t)-\mathrm{id}\|_{L^2}^2& \leq 4t_0^2\sup_{t\leq
  t_0} \|u_t(t)\|_{L^2}^2 +  2\|u_0-\mathrm{id}\|_{L^2}^2 \\ 
& \leq 4 t_0^2 \mathcal{E}_0(0) +  2\|u_0-\mathrm{id}\|_{L^2}^2.
\end{split}
\end{equation}

Further, we observe that thanks to (\ref{ass}), to Korn's inequality
and to
Poincar\'e's inequality,  there exist constants $c, C > 0$ so that:
\begin{equation*}
\begin{split}
c &\mathcal{Z}(t) \leq \\ 
& \int_{\mathbb{R}^3} |u_t|^2 + |\nabla^3u_t|^2  +
2W(\phi,\nabla u)  +
\sum_{i,j,k=1..3} D^2W(\phi,\nabla u) : (\phi_{x_i, x_j,
  x_k}, \nabla u_{x_i, x_j, x_k})^{\otimes 2} ~\mbox{d}x  \\ &
\qquad\qquad\qquad \qquad\qquad\qquad \qquad +
\int_{\mathbb{R}^3} |u-\mathrm{id}|^2 ~\mbox{d}x \leq C \mathcal{Z}(t),
\end{split}
\end{equation*}
as well as: 
\begin{equation*}
\big |\int_{\mathbb{R}^3} \sum_{i,j,k=1..3}
\mathcal{R}_{ijk}\phi_{x_i, x_j, x_k} ~\mathrm{d}x\big |\leq C
\|\mathcal{R}\|_{L^2} \|\nabla^3\phi\|_{L^2} \leq 
C \mathcal{Z}^{3/2}(t) \mathcal{Z}^{1/2}(t) = C \mathcal{Z}^{2}(t),
\end{equation*}
where we estimated each term in (\ref{R2}) by the Cauchy-Schwartz
inequality and noted the appropriate Sobolev embedding.
Consequently, we arrive at:
\begin{equation}\label{a7}
 \forall t\leq t_0\qquad 
\mathcal{E}(t)+\|u-\mathrm{id}\|_{L^2}^2(t) \geq c\mathcal{Z}(t) -
C\mathcal{Z}^{2}(t) \geq c \mathcal{Z}(t),
\end{equation}
provided that $\mathcal{Z}\ll 1$ is sufficiently small on the time
interval we consider. In view of (\ref{a1}), (\ref{a7}) and
(\ref{a4}), we now get:
\begin{equation}\label{a8}
\forall t\leq t_0\qquad  \mathcal{Z}(t) \leq C\Big(\int_0^t \mathcal{Z}^{3/2}(s)~\mbox{d}s +
\mathcal{E}(0) + t_0^2\mathcal{E}_0(0) + \|u_0-\mathrm{id}\|_{L^2}^2\Big). 
\end{equation}

Calling $\bar{\mathcal{Z}} =\sup_{t\in [0,t_0]} \mathcal{Z}(t)$, we have:
\begin{equation}\label{a10}
\bar{\mathcal{Z}}  \leq \big(C t_0 {\bar{\mathcal{Z}}}^{1/2}\Big)
\bar{\mathcal{Z}} +  C\big(\mathcal{E}(0) + t_0^2\mathcal{E}_0(0) + \|u_0-\mathrm{id}\|_{L^2}^2\big),
\end{equation}
which combined with the requirement:  $ C T {\bar{\mathcal{Z}}}^{1/2} \leq \frac12$
yields:
\begin{equation}\label{a12}
\mathcal{Z}(t_0) \leq \bar{\mathcal{Z}}\leq C\big(\mathcal{E}(0) +
t_0^2\mathcal{E}_0(0) + \|u_0-\mathrm{id}\|_{L^2}^2\big). 
\end{equation}
The above clearly implies the Theorem in view of the smallness of initial data in (\ref{small}).
\end{proof}

\section{Proof of Theorem \ref{th1}: Existence of solutions to
  (\ref{maineq}).} \label{sec3}

In this section we construct approximate solutions to the Cauchy problem
(\ref{maineq}) (\ref{initial1} - \ref{initial2}), which satisfy the same a-priori bounds
as in section \ref{abd}.  Given $\epsilon >0$, consider the regularized problem:
\begin{equation}\label{approx}
\left\{\begin{split}
& u_{tt} - \mbox{div}\Big(\partial_F W(\phi, \nabla u)\Big) - \epsilon \Delta u = 0 \\
& \phi_t = \Delta \Big(\partial_\phi W(\phi,\nabla u)\Big)
\end{split}\right.
\end{equation}
with the same initial data as in (\ref{initial1} - \ref{initial2}).

\begin{lemma}\label{th3.1}
Assume that all quantities in (\ref{small}) are sufficiently
small. Then, there exists $T_\epsilon >0$ and a solution $(u^\epsilon,
\phi^\epsilon)$ of (\ref{approx}) (\ref{initial1} - \ref{initial2}) on
$\mathbb{R}^3\times [0, T_\epsilon)$, such that:
\begin{equation*}\label{oldbounds}
\begin{split} 
& u^\epsilon -\mathrm{id} \in L^\infty(0,T;H^4(\mathbb{R}^3)), \quad u^\epsilon_{tt}\in
L^\infty(0,T;H^2(\R^3)), \\
& \phi^\epsilon \in L^\infty(0,T;H^3(\mathbb{R}^3)) \mbox{ \ \ and \ \ } \phi^\epsilon_t \in
L^2(0,T;H^2(\mathbb{R}^3)).
\end{split}
\end{equation*}
\end{lemma}
\begin{proof}
{\bf 1.} Since $\epsilon>0$ is fixed, we drop the superscript
$^\epsilon$ in order to lighten the notation in the next two steps. We proceed by the
Galerkin method. Choose an orthonormal base $\{w^k\}_{k=1}^\infty$ in 
the space $H^4(\R^3,\R^3)$ equipped with the scalar product: 
\begin{equation}\label{e2}
\langle w, \tilde w\rangle_{H^4}=\langle w, \tilde w\rangle_{L^2} + \langle
\nabla^4 w : \nabla^4 \tilde w\rangle_{L^2}.
\end{equation}
Similarly, let $\{v^k\}_{k=1}^\infty$ be an orthonormal basis in
$H^3(\R^3)$ equipped with:
\begin{equation}\label{e3}
 \langle v, \tilde v\rangle_{H^3} = \langle
 v,\tilde v\rangle_{L^2} + \langle \nabla^3 v : \nabla^3 \tilde v\rangle_{L^2}.
\end{equation}
Denote: $W^N=\mbox{span} \{w^1,...,w^N\}$ and $V^N=\mbox{span} \{v^1,...,v^N\}$.

We now introduce the auxiliary scalar products:
\begin{equation}\label{e2a}
\begin{split}
&\langle w, \tilde w\rangle_{W}=\langle w, \tilde w\rangle_{L^2} + \langle
\nabla^3 w : \nabla^3 \tilde w\rangle_{L^2}\qquad \forall w,\tilde w
\in H^4(\R^3,\R^3),\\
& \langle v,\tilde v\rangle_{V} = \langle
 v,(-\Delta)^{-1} \tilde v\rangle_{L^2} + \langle \nabla^3 v :
 (-\Delta)^{-1} \nabla^3 \tilde v\rangle_{L^2} \qquad \forall v, \tilde
 v\in H^3(\R^3,\R).
\end{split}
\end{equation}
Clearly, these products are not equivalent to (\ref{e2}), (\ref{e3}), however
their properties will allow for using the energy estimates of the proof
of Lemma \ref{lemapriori} to prove the regularity of approximate
solutions  $(u^N,\phi^N)$ which we define below.

\smallskip

Let $(u^N, \phi^N)\in W^N\times V^N$ be the solution to:
\begin{equation}\label{e5}
\left\{ \begin{split}
& \big \langle u^N_{tt} - \div \partial_FW(\phi^N,\nabla u^N) - \epsilon \Delta u^N, w^l\big\rangle_W=0, \\
& \big \langle \phi_t^N - \Delta \partial_\phi W(\phi^N,\nabla u^N), v^l\big\rangle_V=0,
\qquad\qquad\qquad \forall l:1\ldots N \\
& u^N(0,\cdot) = \mathbb{P}_{W^N}(u_0), \quad (u^N)_t(0,\cdot) =
\mathbb{P}_{W^N}(u_1), \quad \phi^N(0,\cdot) = \mathbb{P}_{V^N}(\phi_0).
 \end{split}\right.
\end{equation}
By $\mathbb{P}$ we denote here the orthogonal projections on 
appropriate subspaces. The classical theory of systems of ODEs guarantees existence of
solutions to (\ref{e5}) on some time interval $[0,T_N)$. We now prove that
these time intervals may be taken uniform for all sequences
$\{u^N,\phi^N\}$.

\medskip

{\bf 2.} Since $u_t^N\in W^N$ and $\phi_t^N\in V^N$, (\ref{e5}) implies:
\begin{equation}\label{e6}
\begin{split}
& \big\langle u^N_{tt} - \div \partial_F W(\phi^N,\nabla u^N) - \epsilon \Delta u^N, u^N_t\big\rangle_W=0, \\
& \big \langle \phi_t^N - \Delta \partial_\phi W(\phi^N,\nabla u^N), \phi^N_t\big\rangle_V=0.
\end{split}
\end{equation}
Note that the first equation in (\ref{e6}) is equivalent to:
\begin{equation*}\label{e6a}
\begin{split}
& \Big \langle u^N_{tt} - \div \partial_FW(\phi^N,\nabla u^N) -
 \epsilon \Delta u^N, u^N_t\Big \rangle_{L^2} \\
& \quad + \sum_{i,j,k = 1..3} \Big\langle u^N_{x_i, x_j, x_k, tt} - {\rm div}\big(\partial_F^2W(\phi^N, \nabla
  u^N) : \nabla u^N_{x_i, x_j, x_k}\big) -\epsilon \Delta  u^N_{x_i,
  x_j, x_k}, u^N_{x_i,x_j,x_k,t} \Big \rangle_{L^2} \\ 
& = \sum_{i,j,k = 1..3} \Big\langle {\rm div}\big(\partial_F\partial_\phi
  W(\phi^N, \nabla u^N)\phi_{x_i, x_j, x_k}\big), u^N_{x_i,x_j,x_k,t} \Big \rangle_{L^2} +
\big\langle \mathcal{R}_1^{N}: \nabla^3 u^N_t\big\rangle,
\end{split}
\end{equation*}
where by $\mathcal{R}_1^N$ we denote the error terms induced by the
functions $u^N, \phi^N$ as in (\ref{undici}), (\ref{R1}).
Likewise, the second equation in (\ref{e6}) becomes:
\begin{equation*}\label{e6b}
\begin{split}
& \Big\langle \phi_t^N - \Delta \partial_\phi W(\phi^N,\nabla u^N), (-\Delta)^{-1} \phi^N_t\Big\rangle_{L^2}\\
& \quad + \sum_{i,j,k =1..3 } \Big\langle \phi^N_{x_i, x_j, x_k, t} - \Delta\Big((\partial_\phi\partial_F W^N) : \nabla
  u^N_{x_i, x_j, x_k} - (\partial^2_\phi W^N) \phi^N_{x_i, x_j, x_k}\Big),  (-\Delta)^{-1} \phi^N_{x_i, x_j, x_k, t} \Big\rangle_{L^2}\\
& = \sum_{i,j,k =1..3 } \Big\langle \Delta \mathcal{R}^N_{ijk}, (-\Delta)^{-1} \phi^N_{x_i, x_j, x_k, t} \Big\rangle_{L^2},
\end{split}
\end{equation*}
where we used the identity (\ref{diciotto}) and the notation
(\ref{R2}), with the superscript $^N$ indicating that they concern
$u^N$ and $\phi^N$.

\smallskip

Let $\mathcal{E}[u^N, \phi^N](t)$ be as in Lemma \ref{lemapriori}
with $(u,\phi)$ replaced by $(u^N, \phi^N)$, and define:
\begin{equation}\label{e7}
 \mathcal{E}_\epsilon[u^N, \phi^N](t) = \mathcal{E}[u^N,\phi^N](t) +
 \epsilon \big\langle (-\Delta)u^N,u^N\big\rangle_V,
\end{equation}
so that:
\begin{equation*}
 \begin{split}
 \mathcal{E}_\epsilon[\phi^N,u^N](t) = \int_{\mathbb{R}^3} & |u_t^N|^2 + |\nabla^3u^N_t|^2  +
2W(\phi^N,\nabla u^N) +\epsilon |\nabla u^N|^2 \\ & 
+ \sum_{i,j,k=1..3} D^2W(\phi^N,\nabla u^n) : (\phi^N_{x_i, x_j,
  x_k}, \nabla u^N_{x_i, x_j, x_k})^{\otimes 2} + \epsilon |\nabla u^N_{x_i, x_j, x_k}|^2\\
& + 2 \sum_{i,j,k=1..3} \mathcal{R}^N_{ijk}\phi^N_{x_i, x_j, x_k} ~\mathrm{d}x.
\end{split}
\end{equation*}
Following the proof of Lemmas \ref{gaga5} and \ref{lemapriori}, we find the counterpart
of the inequality  (\ref{ventinove}):
\begin{equation}\label{e8}
 \mathcal{E}_\epsilon[u^N, \phi^N](t) \leq C\int_0^t
 \mathcal{Z}^{3/2}[u^N, \phi^N](s)~\mbox{d}s + \mathcal{E}_\epsilon(0), 
\end{equation}
where the constant $C$ is independent from $\epsilon$, and where:
$$ \mathcal{Z}[u^N, \phi^N] (t) = \|u^N_t\|^2_{H^{3}(\mathbb{R}^3)} +
\|\nabla u^N - \mathrm{Id}\|^2_{H^{3}(\mathbb{R}^3)}  + \|\phi^N\|^2_{H^{3}(\mathbb{R}^3)}.  $$
Note that in order to obtain (\ref{e8}) we use only the
 equivalent formulations of (\ref{e6}) above, hence indeed all the
 steps from the proof of Lemma \ref{lemapriori} are valid with
 universal constants. Since the initial data in (\ref{e5}) consists of projections
 of the original data, their norms are uniformly controlled as well.

\medskip

{\bf 3.} We now consider the equivalence of $\mathcal{E}_\epsilon$
with $\mathcal{Z}$. Since for small $\mathcal{Z}$ one has:
$$\int \mathcal{R}^N_{ijk}\phi^N_{x_i, x_j, x_k} ~\mathrm{d}x \leq C
\mathcal{Z}^{3/2}[u^N, \phi^N],$$ we easily see that:
\begin{equation}\label{e8a}
  \mathcal{E}_\epsilon[u^N, \phi^N] \leq C\mathcal{Z}[u^N, \phi^N].
\end{equation}
On the other hand, in view of (\ref{e7}):
\begin{equation}\label{e9}
 \mathcal{E}_\epsilon[u^N, \phi^N] \geq c_\epsilon (\mathcal{Z} -
 C_\epsilon \mathcal{Z}^{3/2}) \geq c_\epsilon \mathcal{Z}[u^N,
 \phi^N], 
\end{equation}
where by $c_\epsilon, C_\epsilon$ we denote positive constants independent of $N$ but
depending on $\epsilon$. By (\ref{e8}) we now arrive at:
\begin{equation*}
 \mathcal{Z}[u^N, \phi^N](t)\leq C_\epsilon \int_0^t
 \mathcal{Z}^{3/2}[u^N, \phi^N](s)~\mbox{d}s + C_\epsilon
 \mathcal{E}_\epsilon(0).
\end{equation*}
Consequently, for $t_{0,\epsilon}$ sufficiently small, we have:
\begin{equation*}
 \sup_{t\leq t_{0,\epsilon}}  \mathcal{Z}[u^N, \phi^N](t) \leq
 C_\epsilon \mathcal{E}_\epsilon(0).
\end{equation*}

The above estimates, in particular (\ref{e8a}) and (\ref{e9}) imply 
the uniform in $N$ boundedness of the following quantities, on their common
interval of existence $[0, T_\epsilon]$:
\begin{equation}\label{e10}
 \begin{split}
&  u^N - \mathrm{id}  \in L^\infty(0,T_\epsilon;H^4(\mathbb{R}^3)),
\qquad u^N_t \in L^\infty(0,T_\epsilon;H^3(\mathbb{R}^3)),\\ 
& \phi^N \in L^\infty(0,T_\epsilon;H^3(\mathbb{R}^3)),\qquad \phi_t^N \in L^2(0,T_\epsilon;H^2(\mathbb{R}^3)),
 \end{split}
\end{equation}
yielding the weak-$*$ convergence in $L^\infty$ as $N\to\infty$ (up to a
subsequence), of the quantities: $u^N - \mathrm{id}$, $u_t^N$,
$\phi^N$, $\phi_t^N$ to the limiting quantities: $u^\epsilon - \mathrm{id}$, $u_t^\epsilon$,
$\phi^\epsilon$, $\phi_t^\epsilon$. Additionally, passing if necessary
to a further subsequence and invoking a diagonal argument, we may also assure that:
$$ \nabla u^N \to \nabla u^\epsilon \quad \mbox{ and } \quad \phi^N
\to \phi^\epsilon \qquad \mbox{point-wise in } \mathbb{R}^3. $$
The Sobolev compact embedding: $ H^1(0,T_\epsilon; H^2(B(R)))
\hookrightarrow \mathcal{C}^\alpha((0,T_\epsilon)\times B(R))$,
valid on any ball $B(R)\subset\mathbb{R}^3$, justifies now that
$\phi^\epsilon \in \mathcal{C}^\alpha((0,T_\epsilon)\times B(R))$. Thus, in particular:
\begin{equation}
 (\partial_\phi W,~ \partial_FW)(\phi^N,\nabla u^N)
 \to (\partial_\phi W,~ \partial_FW)(\phi^\epsilon,\nabla u^\epsilon)
 \quad \mbox{ as } N\to\infty.
\end{equation}

It follows that $(\phi^\epsilon, u^\epsilon)$ is a distributional solution to
(\ref{approx}) (\ref{initial1} - \ref{initial2}). By (\ref{e10}) we obtain the desired
regularity, completing the proof of  Lemma \ref{th3.1}.
\end{proof}

\bigskip 

\noindent {\bf Proof of Theorem \ref{th1} (existence part).}
Let $\phi^\epsilon, u^\epsilon$ be as in Lemma \ref{th3.1}. We first
observe that a common interval of existence of $(\phi^\epsilon, u^\epsilon)$ can be
taken as $[0,T]$ with $T$ prescribed by Theorem \ref{th1}.  This follows through repeating the estimates in section
\ref{abd}, dealing with estimates of the first and the third order
separately, and noting that the $\epsilon$-term appears 
exclusively in $\mathcal{E}_\epsilon$ with a ``good'' sign. Consequently:
\begin{equation}\label{e13}
 \sup_{t\leq t_0} \mathcal{Z}[u^\epsilon, \phi^\epsilon]\leq C(t_0,
 \mbox{initial data}),
\end{equation}
and we see that indeed the solutions $\phi^\epsilon, u^\epsilon$ can be
extended over appropriate  $[0,T]$ with the quantities in (\ref{oldbounds})
enjoying common bounds, independent of $\epsilon$.

The same argument as in the last part of the proof of Lemma
\ref{th3.1} implies now that the weak-$*$ limit (up to a subsequence) of
$(\phi^\epsilon, u^\epsilon)$ yield the desired regular solution $(\phi, u)$ to the original problem
(\ref{maineq}) (\ref{initial1} - \ref{initial2}). Condition (\ref{nonin}) is
automatically satisfied because of the smallness of initial data.
\endproof

\section{Proof of Theorem \ref{th1}: Uniqueness of solutions to (\ref{maineq}).}\label{sec4}

Let $(\phi, u)$ and $(\bar \phi, \bar u)$ be two solutions 
to (\ref{maineq}) with the same initial data. Define:
\begin{equation*}\label{u1}
 (\delta \phi) = \phi - \bar \phi, \qquad (\delta u) = u - \bar u,
\end{equation*}
and observe that:
\begin{equation}\label{u2}
\begin{split}
  & \d_FW(\phi,\nabla u) -  \d_FW(\bar \phi,\nabla \bar u) \\
  & = ~  \d^2_FW(\bar \phi,\nabla \bar u) : \nabla (\delta u) +
 \d_\phi\d_F W (\bar \phi,\nabla \bar \phi) (\delta \phi)  +
D^2\d_F W(\tilde \phi, \nabla \tilde u) : \big((\delta \phi),\nabla (\delta
u)\big)^{\otimes 2}, \\ 
& \d_\phi W(\phi,\nabla u) - \d_\phi W(\bar \phi,\nabla \bar u) \\
&  = ~ \d_\phi^2 W(\bar \phi,\nabla \bar u) (\delta \phi) + \d_\phi\d_F
W(\bar \phi,\nabla \bar u) : \nabla (\delta u) +
D^2 \d_\phi W(\tilde \phi, \nabla \tilde u) : \big((\delta \phi),\nabla (\delta
u)\big)^{\otimes 2}, 
\end{split}
\end{equation}
where $\tilde \phi$ and $\tilde u$ are suitable linear combinations of
$\phi,\bar \phi$ and $u,\bar u$, given by the application of the Taylor formula. 
Subtracting equations (\ref{maineq}) for $(\phi,u)$ and $(\bar
\phi,\bar u)$ and using (\ref{u2}), it follows that:
\begin{equation*}\label{u4}
\begin{split} 
& (\delta u)_{tt} - \div \Big(\d^2_FW(\bar \phi,\nabla \bar u) : \nabla
(\delta u) + \d_\phi\d_F W (\bar \phi,\nabla \bar u) (\delta
\phi)\Big) \\ & \qquad\qquad\qquad\qquad \qquad = 
\div \Big( D^2\d_F W(\tilde \phi, \nabla \tilde u) : \big((\delta \phi),\nabla (\delta
u)\big)^{\otimes 2}\Big), \\ 
& (\delta \phi)_t - \Delta \Big( \d_\phi^2 W(\bar \phi,\nabla \bar u) (\delta
 \phi) + \d_\phi\d_F W(\bar \phi,\nabla \bar u) : \nabla (\delta u) \Big) \\
& \qquad\qquad\qquad\qquad \qquad 
= \Delta \Big( D^2 \d_\phi W(\tilde \phi, \nabla \tilde u) : \big((\delta
\phi),\nabla (\delta u)\big)^{\otimes 2} \Big).
\end{split}
\end{equation*}
We now test the first equation above by $(\delta u)_t$, while the
second equation by $(-\Delta)^{-1} (\delta \phi)_t$, to obtain:
\begin{equation}\label{u6}
\begin{split}
& \frac{1}{2} \int_{\mathbb{R}^3} |\delta u_t|^2 + \d_F^2 W(\bar \phi,
\nabla \bar u) : \big (\nabla (\delta u)\big )^{\otimes 2} ~\mbox{d}x +
\int_{\mathbb{R}^3} \d_\phi\d_F W(\bar \phi,\nabla \bar u) :
(\delta \phi) \nabla (\delta u)_t ~\mbox{d}x \\ &
\qquad\qquad\qquad 
= \int_{\mathbb{R}^3}  D^2\d_F W(\tilde \phi, \nabla \tilde u) : \Big( \big((\delta \phi),\nabla
(\delta u)\big)^{\otimes 2}\otimes \nabla (\delta u)_t\Big) ~\mbox{d}x, \\
& \int_{\mathbb{R}^3} |\nabla (\delta \phi)_t|^2 ~\mbox{d}x + 
\int_{\mathbb{R}^3} \Big( \d_\phi^2 W(\bar \phi,\nabla \bar u) (\delta \phi) + \d_\phi\d_F
W(\bar \phi,\nabla \bar u) : \nabla (\delta u) \Big)(\delta \phi)_t  ~\mbox{d}x \\ &
\qquad\qquad\qquad 
= \int_{\mathbb{R}^3}  D^2 \d_\phi W(\tilde \phi, \nabla \tilde u)  : \big((\delta \phi),\nabla
(\delta u)\big)^{\otimes 2} (\delta \phi)_t ~\mbox{d}x. 
\end{split}
\end{equation}
Consequently:
\begin{equation}\label{u8}
\begin{split}
& \frac 12 \frac{\mbox{d}}{\mbox{d}t} \int_{\mathbb{R}^3}  \d_\phi ^2 W(\bar \phi,\nabla \bar u) (\delta
\phi)^{2}~\mbox{d}x  + \int_{\mathbb{R}^3}  \d_\phi\d_F W(\bar \phi,\nabla \bar
u): (\delta \phi_t)\nabla (\delta u) ~\mbox{d}x \\ &
\qquad\qquad\qquad\qquad\qquad\qquad \qquad
~~ ~~ + \int_{\mathbb{R}^3}
|\nabla (\delta \phi)_t|^2 ~\mbox{d}x \\ & 
=  \int_{\mathbb{R}^3}  D^2 \d_\phi W(\tilde \phi, \nabla \tilde u)  : \big((\delta \phi),\nabla
(\delta u)\big)^{\otimes 2} (\delta \phi)_t ~\mbox{d}x \\ &
\qquad\qquad\qquad\qquad\qquad\qquad \qquad
~~ ~~ + \int_{\mathbb{R}^3} \d_t \Big(\d_\phi^2 W(\bar
\phi,\nabla \bar u)\Big) (\delta \phi)^{2} ~\mbox{d}x. 
\end{split}
\end{equation}

Adding (\ref{u6}) and (\ref{u8}), we arrive at:
\begin{equation*}
\begin{split}
& \frac{\mbox{d}}{\mbox{d}t}\frac{1}{2} \int_{\mathbb{R}^3}   |(\delta u)_t|^2 + \d_F^2
W(\bar \phi, \nabla \bar u) : \big(\nabla (\delta u)\big)^{\otimes 2} + 
\d_\phi^2 W(\bar \phi,\nabla \bar u) (\delta \phi)^{2} \\ &
\qquad\qquad\qquad\qquad \qquad \qquad\qquad\qquad\qquad \quad 
+ 2\d_\phi\d_F W(\bar \phi,\nabla \bar u) : (\delta\phi)\nabla (\delta u) ~\mbox{d}x \\ 
& \leq C \|(\delta \phi)_t, \phi_t,\bar
\phi_t\|_{L_\infty(\mathbb{R}^3)}(t) \cdot \sup_t \|(\delta \phi),\nabla
(\delta u)\|_{L^2(\mathbb{R}^3)}^2(t), 
\end{split}
\end{equation*}
which implies that:
\begin{equation}\label{u10}
\begin{split}
& \sup_t \int_{\mathbb{R}^3}   |\delta u_t|^2 + D^2 W(\bar
 \phi, \nabla \bar u) : \big((\delta \phi), \nabla (\delta u)\big)^{\otimes 2}
 ~\mbox{d}x  \leq  \\ & \qquad\qquad\qquad
C \sup_t \|(\delta \phi),\nabla (\delta u)\|_{L^2}^2 \int_0^t
\|(\delta \phi)_t,\phi_t,\bar \phi_t\|_{L_\infty}(s) ~\mbox{d}s. 
\end{split}
\end{equation}
As before, assumptions on $W$ guarantee that the left hand side in (\ref{u10})
bounds from above the quantity: $\sup_t \|(\delta \phi),\nabla (\delta
u)\|_{L^2}^2$. Since the integral quantity above is small for $t\ll 1$, it
follows by (\ref{u10}) that $(\delta \phi)$ and $\nabla (\delta  u)$
are zero.
\endproof

\section{Proof of Theorem \ref{th2}: The elliptic-parabolic problem (\ref{maineq2}).}\label{sec5}

As in section \ref{sec3}, we first derive an a priori estimate for solutions of 
(\ref{maineq2}), whose existence will follow then via Galerkin's method, in the same 
manner as for the system (\ref{maineq}).

\begin{lemma} \label{aprio}
Assume that $(\phi, u)$ is a sufficiently smooth solution to
(\ref{maineq2})  which remains in a vicinity of
$(0, \mathrm{id})$ for all $t\geq 0$, in the sense that:
$$ \Xi[\phi,\nabla u - \mathrm{Id}]   := \sup_{t\geq 0}( \|\phi\|_{H^2(\mathbb{R}^3)}^2 
+\|\nabla u -Id\|_{H^2}) + c \int_0^\infty
\|\nabla \phi,\nabla^2 u\|_{H^2(\mathbb{R}^3)}^2 ~\mathrm{d}t \ll 1$$
 Then: 
\begin{equation*}
\begin{split}
 \sup_{t\geq 0} \big(\|\nabla u(t) -\mathrm{Id}\|^2_{H^2} + \|u(t) -
 \mathrm{id}\|^2_{L^6}\big) + \Xi[\phi,\nabla u - \mathrm{Id}] \leq 
C\|\phi_0\|^2_{H^2(\mathrm{R}^3)}.
\end{split}
\end{equation*}
\end{lemma}

\begin{proof}
{\bf 1.} 
Observe first the following elementary fact:
\begin{equation}\label{o12}
\|\nabla^2 u(t)\|_{H^1(\mathbb{R}^3)} \leq C \|\nabla \phi(t)\|_{H^1(\mathbb{R}^3)}.
\end{equation}
To see (\ref{o12}), consider the first equation in $(\ref{maineq2})$:
\begin{equation}\label{o13}
 \partial^2_{F}W(\phi,\nabla u) : \nabla u_{x_i} =
- \partial_{\phi}\partial_{F}W(\phi,\nabla u)\phi_{x_i} \qquad i=1\ldots 3. 
\end{equation}
Condition (\ref{ass}) and Korn's inequality imply that the system
(\ref{o13}) is elliptic, hence its solutions (normalised so that  $\nabla
u - \mbox{Id} \in L^6(\mathbb{R}^3)$) obey:
\begin{equation*}\label{o14}
 \|\nabla^2 u\|_{L^p(\mathbb{R}^3)} \leq C\|\nabla \phi\|_{L^p(\mathbb{R}^3)} \qquad p=2,4.
\end{equation*}
Differentiating (\ref{o13}) with respect to $x$ leads further to:
$$\|\nabla^3 u\|_{L^2} \leq C(\|\nabla^2 \phi\|_{L^2} + \|\nabla
\phi,\nabla^2 u\|_{L^4}^2)\leq C(\|\nabla^2 \phi\|_{L^2} + \|\nabla
\phi\|_{L^4}^2) \leq C(\|\nabla^2 \phi\|_{L^2} + \|\nabla
\phi\|_{H^1}^2),$$
proving (\ref{o12}) in view of the assumption in the Lemma.

We also observe the resulting control of pointwise smallness of  
$\phi$ and $(\nabla u - \mbox{Id})$, by the Sobolev embedding:
\begin{equation}\label{o23}
\sup_{t\geq 0} \|\phi, \nabla u - \mbox{Id}\|_{L^\infty} \leq C \sup_{t\geq 0}
\|\phi, \nabla u - \mbox{Id}\|_{H^2} \leq C\Xi^{1/2} \ll 1.
\end{equation}

\medskip

{\bf 2.}  
Testing the first equation in (\ref{maineq2}) by $u_t$ and the second one
 by $\psi_t = (-\Delta)^{-1} \phi_t$, we obtain the  energy estimate, as in Lemma \ref{gaga5}: 
\begin{equation*}
 \frac{\mbox{d}}{~\mbox{d}t}\int_{\mathbb{R}^3} W(\phi,\nabla u) ~\mbox{d}x + \int_{\mathbb{R}^3}
 |\nabla (-\Delta)^{-1} \phi_t|^2~\mbox{d}x =0. 
\end{equation*}
To derive the second energy estimate we proceed slightly
differently. Differentiating (\ref{maineq2}) in a spatial direction
$x_i\in \{x_1,x_2,x_3$\}, we get:
\begin{equation}\label{o3}
\begin{split}
& \div \big(\partial_{F}^2W(\phi,\nabla u):\nabla u_{x_i}
  +\partial_{F}\partial_\phi W(\phi,\nabla u)\phi_{x_i}\big)=0, \vspace{1mm}\\
&\phi_{x_i, t}=\Delta \big(\partial_{\phi}\partial_{F}W(\phi, \nabla
u) : \nabla u_{x_i} + \partial^2_{\phi}W(\phi,\nabla u)\phi_{x_i}\big).
\end{split}
\end{equation}
Now, testing the first equation above by $u_{x_i}$, testing the second one by
$\psi_{x_i} = (-\Delta)^{-1} \phi_{x_i}$, and summing up the results, yields:
\begin{equation}\label{o4}
\frac{1}{2} \frac{\mbox{d}}{~\mbox{d}t} \int_{\mathbb{R}^3} |\nabla(-\Delta)^{-1}
\phi_{x_i}|^2 ~\mbox{d}x + \int_{\mathbb{R}^3} D^2W(\phi,\nabla u) : (\phi_{x_i},\nabla
u_{x_i})^{\otimes 2} ~\mbox{d}x=0 .
\end{equation}
Consequently, thanks to (\ref{ass}), the strict convexity of $W$ implies:
$$\frac{1}{2}\frac{\mbox{d}}{\mbox{d}t} \int_{\R^3}|\nabla^2\psi|^2 +
  \frac{\gamma}{2}\int_{\R^3} \big(|\nabla\phi|^2 + \sum_{i=1}^3 | (\mbox{sym}\nabla
  u_{x_i})|^2\big)\leq 0.$$ 
Using Korn's inequality and integrating in time we see that:
\begin{equation}\label{o4a}
 \sup_{t>0} \int_{\R^3} \phi^2 ~\mbox{d}x + c \int_0^\infty \int_{\R^3}
 (|\nabla \phi|^2 + |\nabla^2 u|^2) ~\mbox{d}x \mbox{d}t\leq
 C\|\phi_0\|^2_{L^2(\mathbb{R}^3)}. 
\end{equation}

\medskip

{\bf 3.} We now differentiate (\ref{o3}) in a spatial direction $x_j \in \{x_1,x_2,x_3\} $, getting:
\begin{equation}\label{o5}
 \begin{array}{l}
  \div \big(\d_{F}^2W(\phi,\nabla u) : \nabla u_{x_i,x_j}
  +\d_{\phi}\d_FW(\phi,\nabla u) \phi_{x_i,x_j}\big)=\div \mathcal{R}_1, \\[5pt]
\phi_{x_i,x_j,t} - \Delta (\partial^2_{\phi}W(\phi,\nabla
u)\phi_{x_i,x_j} + \partial_{\phi}\partial_F W(\phi,\nabla u) : \nabla
u_{x_i,x_j}\big)=\Delta \mathcal{R}_2,
 \end{array}
\end{equation}
where the error terms $\mathcal{R}_1$ and $\mathcal{R}_1$ have the
following structure (we suppress the distinction between different $x_i, x_j$):
\begin{equation}\label{o6}
 \mathcal{R}_1 , \mathcal{R}_2 \sim D^3W(\phi,\nabla u) : \big(
 (\nabla u_x)^{\otimes 2} + (\nabla u_x) \phi_x + (\phi_x)^2\big).
\end{equation}
Integrating (\ref{o5}) by parts against $u_{x_i, x_j}$ and $(-\Delta)^{-1}\phi_{x_i, x_j}$, respectively, it follows that:
\begin{equation}\label{o7}
\begin{split}
\frac{\mbox{d}}{~\mbox{d}t} \int_{\R^3} |\nabla
 \psi_{x_i,x_j}|^2 ~\mbox{d}x ~  + c\int_{\R^3} D^2 W(\phi,\nabla u)&
 : (\phi_{x_i,x_j},\nabla u_{x_i,x_j})^{\otimes 2} ~\mbox{d}x \\ &
 \qquad  \leq C \int_{\R^3} |\nabla^2 u|^4 + |\nabla \phi|^4 ~\mbox{d}x,
\end{split}
\end{equation}
because of (\ref{o6}) and:
\begin{equation*}
\begin{split}
 |\int_{\R^3} (\div \mathcal{R}_1) u_{x_i,x_j} ~\mbox{d}x| + &
 |\int_{\R^3} (\Delta \mathcal{R}_2) (-\Delta)^{-1} \phi_{x_i,x_j}
 ~\mbox{d}x| \\ & \qquad 
\leq \epsilon \| \nabla u_{x_i,x_j}, \phi_{x_i,x_j}\|_{L^2(\mathbb{R}^3)}^2 +
C\|\mathcal{R}_1,\mathcal{R}_2\|^2_{L^2(\mathbb{R}^3)}. 
\end{split}
\end{equation*}
Differentiating (\ref{o5}) further, we obtain:
\begin{equation*}
\begin{array}{l}
\div \big(\partial_{F}^2W(\phi,\nabla u) : \nabla u_{x_i,x_j,x_k}
  + \partial_{\phi}\partial_FW(\phi,\nabla u)\phi_{x_i,x_j,x_k}\big) = \div \mathcal{R}_3, \\[5pt] 
\phi_{x_i,x_j,x_k,t} - \Delta \big(\partial^2_{\phi}W(\phi,\nabla
u)\phi_{x_i,x_j,x_k} + \partial_{\phi}\partial_F W(\phi,\nabla u) : \nabla
u_{x_i,x_j,x_k}\big)=\Delta \mathcal{R}_4, 
 \end{array}
\end{equation*}
where, as before:
\begin{multline*}
\mathcal{R}_3 , \mathcal{R}_4 \sim D^3W(\phi,\nabla u) : \big(
\nabla u_{xx} \otimes \nabla u_x + \nabla u_{xx}\phi_x +  \nabla u_x\phi_{xx} +\phi_x\phi_{xx}\big)\\
+  D^4W(\phi,\nabla u) : \big(
(\nabla u_x)^{\otimes 3} + (\nabla u_x)^{\otimes 2}\phi_x + \nabla
u_{x}(\phi_x)^2 + \nabla u_x(\phi_{x})^2 + (\phi_x)^2\big).
\end{multline*}
Testing by $u_{x_i,x_j,x_k}$ and $\psi_{x_i,x_j,x_k}$, we find:
\begin{equation}\label{o10}
\begin{split}
\frac{\mbox{d}}{~\mbox{d}t} \int_{\R^3} |\nabla \psi_{x_i,x_j,x_k}|^2 &~\mbox{d}x + 
c\int_{\R^3} D^2 W(\phi,\nabla u) : (\phi_{x_i,x_j,x_k},\nabla
u_{x_i,x_j,x_k})^{\otimes 2} ~\mbox{d}x \\ & \leq
C \int_{\R^3} |\nabla^3 u|^4 + |\nabla^2 u|^4 + |\nabla\phi|^4 + |\nabla^2\phi|^4 +  |\nabla^2 u|^6 
+ |\nabla\phi|^6 ~\mbox{d}x. 
\end{split}
\end{equation}
Summing (\ref{o4}), (\ref{o7}), (\ref{o10}), integrating the
result in time in the same manner as in (\ref{o4a}), and recalling
(\ref{o12}), we obtain:
\begin{equation}\label{o11}
\begin{split}
 \Xi[\phi,\nabla u - \mbox{Id}]   & \leq 
C \int_0^\infty \int_{\mathbb{R}^3} |\nabla^2 u|^4 + |\nabla
\phi|^4 + |\nabla^2 u|^6 + |\nabla \phi|^6  
 \\ & \qquad \qquad
 \qquad\qquad + |\nabla^3 u|^4  + |\nabla^2 u|^4 + |\nabla^2\phi
 |^4 ~\mbox{d}x \mbox{d}t + C\|\phi_0\|^2_{H^2(\mathbb{R}^3)} \\ &
\leq C \int_0^\infty \int_{\mathbb{R}^3} |\nabla \phi|^4 + |\nabla \phi|^6  
 + |\nabla^2 \phi|^4 ~\mbox{d}x \mbox{d}t + C\|\phi_0\|^2_{H^2(\mathbb{R}^3)}.
\end{split}
\end{equation}
We further have:
\begin{equation*}
\begin{split}
 \int_0^\infty \int_{\mathbb{R}^3} |\nabla &\phi|^4 + |\nabla \phi|^6  
 + |\nabla^2 \phi|^4 ~\mbox{d}x \mbox{d}t \\ & \leq C \sup_{t\geq
   0}\big(\|\nabla \phi\|_{L^\infty}^2 + \|\nabla \phi\|_{L^\infty}^4
 + \|\nabla^2 \phi\|_{L^\infty}^2\big) \int_0^\infty
 \|\nabla\phi\|_{H^1(\mathbb{R}^3)}^2~\mbox{d}t \leq C(\Xi^2  + \Xi^3)
\end{split}
\end{equation*}
Consequently, (\ref{o11}) becomes:
$ \Xi \leq C(\Xi^2+\Xi^3) + C\|\phi_0\|_{H^2}^2$.
By the assumed smallness of $\Xi[\phi,\nabla u - \mathrm{Id}]$, we see that:
\begin{equation}\label{o21}
 \Xi \leq 2C\|\phi_0\|_{H^2}^2.
\end{equation}

\medskip

{\bf 4.}  We now conclude the proof of the a-priori bound. 
Test (\ref{o13}) by $(u-\mbox{id})$ to get:
\begin{equation*}
\begin{split}
 \int_{\R^3} \d_{F}^2 W(\phi,&\nabla u) : (\nabla u - \mbox{Id})^{\otimes 2} \mbox{d}x \\
& \leq C \int_{\R^3}  |\phi||\nabla u - \mbox{Id}| +
|u-\mbox{id}|\big(|\nabla u - \mbox{Id}| + |\phi|\big) \big(|\nabla
\phi|+|\nabla^2 u|\big) ~\mbox{d}x \\ & \leq
C\|\phi\|^2_{L^2} + (\epsilon + C\Xi^{1/2}) \|\nabla u - \mbox{Id}\|^2_{L^2} 
\end{split}
\end{equation*}
Indeed, by (\ref{o23}) and (\ref{o12}):
\begin{equation*}
\begin{split}
\int_{\R^3}  
|u-\mbox{id}|\big(|\nabla u &- \mbox{Id}| + |\phi|\big) \big(|\nabla
\phi|+|\nabla^2 u|\big) ~\mbox{d}x \\ & \leq 
C\|u-\mbox{id}\|_{L^6} \|\nabla u - \mbox{Id}\|_{L^2} \|\nabla \phi,\nabla^2\phi\|_{L^3} +
C\|u-\mbox{id}\|_{L^6} \|\phi\|_{L^2} \|\nabla
\phi,\nabla^2\phi\|_{L^3}  \\ & \leq
C\|\nabla u-\mbox{Id}\|^2_{L^2} \Xi^{1/2} + 
C\|\phi\|_{L^2}^2 + C\Xi \|\nabla u-\mbox{id}\|^2_{L^2} 
\end{split}
\end{equation*}
Thus, we obtain the bound on $\|\nabla u - \mbox{Id}\|_{L^2}$, and subsequently on
$\|u-\mbox{id}\|_{L^6}$. 
\end{proof}

\bigskip 

\noindent {\bf A proof of Theorem \ref{th2}.}
Given $(\bar\phi, \bar u)$, consider the following problem which is
 the linearization of (\ref{maineq2}) at $(0,\mbox{id})$:
\begin{equation}\label{ex1}
\begin{split}
& \div \Big(\partial^2_{F}W(0,\mbox{Id}) (\nabla u -\mbox{Id})
+ \partial_{\phi}\partial_F W(0,\mbox{Id})\phi\Big) = \div A, \\
& \phi_t - \Delta \Big(\partial^2_{\phi} W(0,\mbox{Id})\phi
+ \partial_{\phi}\partial_F W(0,\mbox{Id})(\nabla u - \mbox{Id})\Big) = \Delta B,
\end{split}
\end{equation}
where:
\begin{equation}\label{ex2}
\begin{split}
& A = \partial^2_{F}W(0,\mbox{Id})(\nabla \bar u-\mbox{Id})
+ \partial_{\phi}\partial_F W(0,\mbox{Id})\bar \phi- \partial_FW(\bar \phi,\nabla \bar u), \\
& B = \partial_\phi W(\bar \phi,\nabla \bar u)-\partial^2_{\phi}W(0,\mbox{Id})\bar \phi + 
\partial_{\phi}\partial_F W(0,\mbox{Id})(\nabla \bar u - \mbox{Id}).
\end{split}
\end{equation}
Let ${\mathcal T}$ be its solution operator, so that $\mathcal{T}[\bar
\phi,\bar u] = (\phi, u)$. We will prove that $\mathcal{T}$ has a
fixed point in the space $X$, where:
\begin{equation*}
\begin{split}
X = \big\{ (\phi, u); ~~& \phi \in L^\infty(\R_+;H^2(\mathbb{R}^3)),  ~\nabla(\nabla u - \mbox{Id}) \in
 L^\infty(\R_+;H^1(\mathbb{R}^3)), \\ & \nabla \phi \in L^2(\R_+;H^2(\mathbb{R}^3))), ~\nabla(\nabla u -
 \mbox{Id}) \in L^2(\R_+;H^2(\mathbb{R}^3))) \big\}. 
\end{split}
\end{equation*}

Note first that the well-posedness of the system (\ref{ex1}) 
follows by the Galerkin method in exactly the same manner as in 
section \ref{sec3}, under the regularity of the right hand side:
\begin{equation*}
 A \in L^\infty(\R_+;H^2(\mathbb{R}^3)) \cap
 L^2(\R_+;H^3(\mathbb{R}^3)), \quad B \in L^2(\R_+;H^3(\mathbb{R}^3))
\end{equation*}
Approximative spaces are constructed for $\phi\in H^3(\R^3)$ and for $u-\mbox{id}$ such that 
$\nabla u - \mbox{Id} \in H^3(\mathbb{R}^3)$. We leave this construction
to the reader and note that it is simpler than
the one for the system (\ref{maineq}). As in the proof of Lemma
\ref{aprio}, solutions to (\ref{ex1}) then satisfy:
\begin{equation*}
\begin{split}
\sup_{t\geq 0} \int_{\R^3} \phi^2 ~\mbox{d}x  + \sum_{i} \int_0^\infty \int_{\R^3}
D^2W(0,\mbox{Id}) &: (\phi_{x_i},\nabla u_{x_i})^{\otimes 2} ~\mbox{d}x
\mbox{d}t \\ & \leq C\|\nabla A,\nabla
B\|^2_{L^2(\R_+;L^2(\R^3))} + C\|\phi_0\|^2_{L^2},
\end{split}
\end{equation*}
which is obtained by testing with $u_{x_i}$ and $(-\Delta)^{-1}\phi_{x_i}$.
Similarly, the second and third derivatives bounds eventually  yield:
\begin{equation}\label{ex3}
\begin{split}
\sup_{t\geq 0} \|\phi\|^2_{H^2(\mathbb{R}^3)} + \|\nabla \phi, \nabla(&\nabla u -
\mbox{Id})\|_{L^2(\R_+;H^2(\mathbb{R}^3))}^2 \\ & \leq C\|\nabla A,\nabla B\|_{L^2(\R_+;H^2(\mathbb{R}^3))}^2
+ C\|\phi_0\|^2_{H^2(\mathbb{R}^3))},
\end{split}
\end{equation}
while:
\begin{equation*}
\sup_{t\geq 0} \|\nabla u - \mbox{Id}\|_{H^2(\mathbb{R}^3)}^2\leq C\|A\|_{L^\infty(\R_+;H^2(\mathbb{R}^3))}^2.
\end{equation*}
Directly from (\ref{ex2}) we observe that: 
\begin{equation*}
\begin{split}
& \|A\|_{L^\infty(\R_+;H^2(\mathbb{R}^3)) \cap L^2(\R_+;H^3(\mathbb{R}^3))} \leq C~\Xi[\bar \phi,\nabla \bar u - \mbox{Id}]\\
& \|B\|_{ L^2(\R_+;H^3(\mathbb{R}^3))} \leq C~\Xi[\bar \phi,\nabla
\bar u - \mbox{Id}]
\end{split}
\end{equation*}
provided the quantity $\Xi$ is small. Then, by (\ref{ex3}):
\begin{equation*}
 \Xi[\phi,\nabla u -\mbox{Id}] \leq C \Xi[\bar \phi,\nabla \bar u - \mbox{Id}]^2 +C_0\|\phi_0\|_{H^2(\mathbb{R}^3)}^2. 
\end{equation*}
Based on the considerations from the part about the a priori bound we observe that:
\begin{equation*}
 \Xi[\phi,\nabla u - \mbox{Id}] \leq 2 C_0\|\phi_0\|^2_{H^2(\R^3)},
\end{equation*}
provided that $\Xi$ is sufficiently small.
Hence the operator $\mathcal{T}$ maps a ball  $\mathcal{B}\subset X$ with a
sufficiently small radius, into itself.
Observe further that $\mathcal{T}$ is a contraction over
$\mathcal{B}$, whose fixed point yields the unique solution to the
system (\ref{maineq2}). Theorem \ref{th2} is proved.
\endproof

\section{Appendix: Proof of Proposition \ref{prop}.}

The first condition in (\ref{ass}) is obvious. A
direct calculation shows that:
$$DW_1(\phi, F) : (\tilde\phi,\tilde F)  = \phi\tilde\phi + \langle 
DW_0(FB(\phi)) : \tilde \phi F B'(\phi)\rangle + \langle DW_0(FB(\phi))
: \tilde F B(\phi)\rangle,$$
which implies the second condition in (\ref{ass}). Further:
\begin{equation*}
\begin{split}
D^2W_1(0,\mbox{Id}) : (\tilde\phi,\tilde
F)^{\otimes 2} & = |\tilde \phi|^2
+ D^2W_0(\mbox{Id}) : (\tilde\phi B'(0))^{\otimes 2} \\ &
\qquad + 2
D^2W_0(\mbox{Id}) : (\tilde B'(0)\otimes \tilde F) + D^2W_0(\mbox{Id}) :
{\tilde F}^{\otimes 2} \\ &  = |\tilde \phi|^2 + D^2W_0(\mbox{Id}) : (\tilde F
+\tilde \phi B'(0))^{\otimes 2} \\ & \geq |\tilde\phi|^2 + c\left|\mbox{sym }
\tilde F + \tilde \phi B'(0)\right|^2,
\end{split}
\end{equation*}
where we concluded from (\ref{elastic_dens}) that $DW_0(\mbox{Id}) = 0$ and that
$D^2W_0(\mbox{Id})$ is positive definite on symmetric matrices.
We also note that: $ D^2W_2(0,\mbox{Id}) = D^2W_1(0,\mbox{Id}).$
To conclude the proof, it is hence enough to show that:
\begin{equation}\label{want}
|\tilde\phi|^2 + |\mbox{sym } \tilde F + \tilde\phi B'(0)|^2 \geq c
(|\tilde\phi|^2 + |\mbox{sym } \tilde F|^2 ),
\end{equation}
for all $\tilde\phi$ and $\tilde F$. Expanding the square in the left
hand side, dividing by $|\tilde \phi|$ and collecting terms, this is
equivalent to: 
\begin{equation*}
(1-c + |B'(0)|^2) + (1-c) |\mbox{sym} (\frac{1}{\tilde\phi}\tilde
F)|^2 + 2\langle \mbox{sym} (\frac{1}{\tilde\phi}\tilde F) :
B'(0)\rangle \geq 0,
\end{equation*}
which becomes:
\begin{equation*}
1-c -\frac{c}{1-c} |B'(0)|^2 +  |\sqrt{1-c}
\mbox{ sym} \big(\frac{1}{\tilde\phi}\tilde F\big) + \frac{1}{\sqrt{1-c}} B'(0)|^2 \geq 0. 
\end{equation*}
The above inequality follows from:  $1-c -
\frac{c}{1-c} |B'(0)|^2 >0$, which is true whenever $c>0$  is sufficiently small.
\endproof

\end{document}